\documentclass[12pt]{amsart}

\usepackage{lineno}
\usepackage[utf8]{inputenc}
\usepackage{comment}
\usepackage[english]{babel}
\usepackage[T1]{fontenc}
\usepackage{amsmath,amsfonts,amssymb,amsthm}
\usepackage{amscd}
\usepackage[a4paper,nomarginpar]{geometry}
\usepackage{epsfig,graphicx,color}
\usepackage[colorlinks=true, linkcolor=blue, urlcolor=red, citecolor=blue, hyperindex,backref]{hyperref}
\usepackage[all]{xy}

\theoremstyle{plain}

\newtheorem{theorem}{Theorem}[section]

\newtheorem{proposition}[theorem]{Proposition}
\newtheorem{corollary}[theorem]{Corollary}

\newtheorem{definition}[theorem]{Definition}

\newtheorem{remark}[theorem]{Remark}
\newtheorem{question}[theorem]{Question}

\newcommand{\N}{\mathbb{N}}  
\newcommand{\Z}{\mathbb{Z}}  
\newcommand{\R}{\mathbb{R}}  

\newcommand{\eps}{\varepsilon} 

\newcommand{\diam}{\operatorname{diam}}
\newcommand{\stn}{\operatorname{st_n}}
\newcommand{\stomega}{\operatorname{st_\omega}}

\begin{document}

\title[Shadowing in the hyperspace of continua]{Shadowing in the hyperspace of continua}

\author{Bernardo Carvalho and Udayan Darji}

\date{\today}
\thanks{2020 \emph{Mathematics Subject Classification}: Primary 37D10, 37B40, 37B45, Secondary 37B99.}
\keywords{Shadowing, hyperspace, continua, induced map.}

\begin{abstract}
We discuss whether classical examples of dynamical systems satisfying the shadowing property also satisfy the shadowing property for the induced map on the hyperspace of continua, obtaining both positive and negative results. We prove that transitive Anosov diffeomorphisms, or more generally continuum-wise hyperbolic homeomorphisms, do not satisfy the shadowing property for the induced map on the hyperspace of continua. We prove that dendrite monotone maps satisfy the shadowing property if, and only if, their induced map on the hyperspace of continua also satisfies it. We give some algorithms that show that there are abundant dynamical systems satisfying the shadowing property with dendrites (compact metric trees) as the underlying space. As a consequence, we show that the universal dendrite of order $n$ admits a homeomorphism with the shadowing property.
\end{abstract}

\maketitle

\section{Introduction}

The dynamical systems theory studies the behavior of orbits generated by some evolution law. We can think of a discrete dynamical system as a map $f\colon X\to X$ where $(X,d)$ is a compact metric space. In this case, the orbit of a point $x\in X$ is the sequence $(f^n(x))_{n\in\N}$, where $f^n$ denotes the composition of $f$ with itself $n$ times. A central question in the theory is: what happens with orbits $(f^n(x))_{n\in\N}$ when $n\to\infty$? In many cases, it is necessary to approximate or truncate the orbit, yielding what are called \emph{pseudo-orbits}. They do not describe the real evolution of the system, therefore it would be useful to obtain a real orbit (hence the real evolution of the system) following them through time. In this way, the pseudo-orbit errors can be disregarded. This property in a dynamical system is called the \emph{shadowing property} (see definition \ref{shadowing}). Its importance comes from distinct research directions, such as the stability theory \cite{Walterspotp}, recurrence theory \cite{Conley}, chaotic and hyperbolic dynamics \cite{Anosov}, \cite{Bowen}, \cite{SMALE67}, and it is a significant part of the qualitative study of dynamical systems that contains several important and deep results (see the monographs \cite{Palmer} and \cite{PilyuginShadowingBook1999}).

Many distinct notions of shadowing properties were introduced and discussed extensively in the literature. They either assume distinct notions of pseudo-orbits or distinct notions of shadowing orbits (see for example \cite{ACCV}, \cite{C}, \cite{CK}, \cite{ENP}, \cite{KO}, \cite{PilyuginLimsha}, \cite{PilyuginVarsha}, \cite{PRS}, \cite{PT}, \cite{Sakai}, \cite{Vleck}). In this article, we discuss pseudo-orbits where elements are not necessarily single points but are sets of the space (not necessarily small). We restrict ourselves to the case the sets are continua (compact and connected subsets) obtaining for every pseudo-orbit of continua, a continuum shadowing it through time. This can be easily explained using the induced map on the hyperspace of continua as follows. 

Let $f\colon X\to X$ be a continuous map of a compact metric space $(X,d)$, $\mathcal{C}(X)$ be the set of all subcontinua of $X$, and $C(f)\colon\mathcal{C}(X)\to\mathcal{C}(X)$ be the map defined by $C(f)(A)=f(A)$. We call $C(f)$ the induced map of $f$ on the hyperspace of continua $\mathcal{C}(X)$. Endow $\mathcal{C}(X)$ with the Hausdorff distance $d_H(F,G)=\inf\{\eps>0; F\subset B(G,\eps) \,\,\text{and}\,\, G\subset B(F,\eps)\}$, where for a compact set $A$, $B(A,\eps)=\{x\in M; d(x,A)<\eps\}$. Thus, we obtain a compact metric space where $C(f)$ is a continuous self-map. In the set $\mathcal{C}(X)$, points are continua of $X$ so the shadowing property of the system $(\mathcal{C}(X),C(f))$ describes precisely the property of shadowing pseudo-orbits of continua we described above.

The study of dynamics in hyperspaces comes from the idea of understanding collective dynamics.  Hyperspace dynamics has applications in various disciplines of sciences, economics, and engineering (see, for example, \cite{DD}, \cite{MLR}, \cite{Rohde}, \cite{Rubinov}, \cite{Handwheel}, \cite{Electrons}, \cite{Zaslavski}). For example, if one considers population dynamics, then the hyperspace dynamics determines collective behavior of sub-populations as an aggregate as opposed to behavior of an individual agent. More precisely, given a dynamical system $(X,f)$ its hyperspace dynamics is the dynamical system $(2^X, 2^f)$ where $2^X$ is the hyperspace of compact subsets of $X$ and $2^f$ is the induced map defined by $2^f(A)=f(A)$. There is extensive literature that investigates similarities and differences between the dynamics of a map and its hyperspace dynamics. The properties studied include topological entropy, sensitivity to initial conditions, recurrence properties, and specification-like properties (see \cite{KwietniakOprocha}, \cite{CanovasLopez}, \cite{Banks}, \cite{MaHouLiao}, \cite{ZhangZenLiu}, \cite{CamargoJavierGarcia}, \cite{AcostaGerardoIllanes}, \cite{Fedeli}, \cite{FloresCano}, \cite{RomanFlores}, \cite{BauerSigmund}, and \cite{ACCC}).  

Regarding the shadowing property, it is proved in \cite{FernandezGood} that $f$ has the shadowing property if, and only if, $2^f$ has the shadowing property. It is easy to prove that if $C(f)$ has the shadowing property, then $f$ has the shadowing property \cite{FernandezGood}. As far as the converse is concerned, in \cite{ArbietoBohorquez} it is shown that if $f\colon\mathbb{S}^1\to\mathbb{S}^1$ is a Morse-Smale diffeomorphism, then $C(f)$ does not satisfy the shadowing property. Morse-Smale diffeomorphisms of $\mathbb{S}^1$ are the simplest possible dynamical systems where there are (up to a finite iterate) only a finite number of fixed points that are either sinks or sources and the orbit of every other point converges to these fixed points in the future and in the past. This was generalized in \cite[Proposition 9]{ArbietoBohorquez} to the North-pole/South-pole diffeomorphism of $\mathbb{S}^2$ but it is still open in the case of general Morse-Smale diffeomorphisms (even on the Sphere). 

Thus, it is interesting to understand what are the maps that induce the shadowing property in $\mathcal{C}(X)$ and the ones that do not. In this work, we discuss this problem for a few classes of systems satisfying the shadowing property, obtaining both positive and negative results. 

Compact metric trees, also known as dendrites, arise in a variety of situations. Some of the earliest results concerning dendrites were obtained by Wa\.zewski who constructed and investigated properties of the universal dendrites. Dendrites appear as Julia sets of quadratic polynomials in the plane \cite{CarlesonGamelinBook}. Recently, inspired by work of  Crovisier and Pujals \cite{CrovisierPujals2018},  Boronski and Stimac \cite{BoronskiStimac2023} showed that strange attractors of certain Henon-like and Lozi-like maps are conjugate to inverse limits of maps on dendrites.   Dendrites have also enjoyed attention from descriptive set theory, in particular from Fra\"iss\'e theory. For example, Codenotti and Kwiatkowska \cite{Codenotti2024} constructed generalized  Wa\.zewski dendrites as projective Fra\"iss\'e limits. Duchesne and Monod \cite{DuchesneEtalTAMS2019} investigated conjugacy class of the group of homeomorphisms of the universal dendrite whose branch points have cardinality $n$, $n >2$. In particular, they proved that the automorphism group of such universal dendrite has a comeager conjugacy class if and only if $n = \omega$.  We prove that if $f\colon D\to D$ is a monotone map on a dendrite $D$, then $C(f)$ has the shadowing property if, and only if, $f$ has the shadowing property.

A classical result in hyperbolic dynamics is that hyperbolic sets satisfy the shadowing property. This was proved independently by Anosov \cite{Anosov} and Bowen \cite{Bowen}. In particular, Anosov diffeomorphisms, where the whole ambient manifold is a hyperbolic set, satisfy the shadowing property. The literature about Anosov diffeomorphisms is extensive and more information can be found in \cite{SMALE67}, \cite{Franksthesis}, \cite{Franks}. The shadowing property is not a property restricted to hyperbolic systems and classifying  systems satisfying the shadowing property is one of the challenging problems in dynamical systems. Some generalizations of hyperbolicity also imply the shadowing property. This is the case, for example, of the continuum-wise hyperbolicity considered in \cite{ARTIGUE2024512} that generalizes hyperbolicity with respect to the continuum theory. We prove that if $f\colon M\to M$ is a transitive Anosov diffeomorphism of a closed manifold, or more generally a continuum-wise hyperbolic homeomorphism of a Peano continuum, then the induced map $C(f)$ does not satisfy the shadowing property. 

\section{Continuum-wise expansive/hyperbolic homeomorphisms}

The main result of this section is Theorem \ref{cw-hyp} where we prove that $C(f)$ does not have the shadowing property when $f$ is a transitive cw-hyperbolic homeomorphism (this includes the case of Anosov diffeomorphisms). Before stating the results of this section, we state all necessary definitions.

\begin{definition}[Shadowing]\label{shadowing}
We say that a homeomorphism $f\colon X\rightarrow X$ of a compact metric space $(X,d)$ satisfies the \emph{shadowing property} if given $\varepsilon>0$ there is $\delta>0$ such that for each sequence $(x_n)_{n\in\mathbb{Z}}\subset X$ satisfying
$$d(f(x_n),x_{n+1})<\delta \,\,\,\,\,\, \text{for every} \,\,\,\,\,\, n\in\mathbb{Z}$$ there is $y\in X$ such that
$$d(f^n(y),x_n)<\varepsilon \,\,\,\,\,\, \text{for every} \,\,\,\,\,\, n\in\mathbb{Z}.$$
In this case, we say that $(x_k)_{k\in\mathbb{Z}}$ is a $\delta-$pseudo orbit of $f$ and that $(x_n)_{n\in\mathbb{Z}}$ is
$\varepsilon-$shadowed by $y$.
\end{definition}

\begin{definition}[Local stable/unstable sets/continua]
For each $x\in X$ and $c>0$, let 
$$W^s_{c}(x):=\{y\in X; \,\, d(f^k(y),f^k(x))\leq c \,\,\,\, \textrm{for every} \,\,\,\, k\geq 0\}$$
be the \emph{c-stable set} of $x$ and
$$W^u_{c}(x):=\{y\in X; \,\, d(f^k(y),f^k(x))\leq c \,\,\,\, \textrm{for every} \,\,\,\, k\leq 0\}$$
be the \emph{c-unstable set} of $x$. Denote by $C^s_c(x)$ the $c$-stable continuum of $x$, that is the connected component of $x$ on $W^s_{c}(x)$, and by $C^u_c(x)$ the $c$-unstable continuum of $x$, that is the connected component of $x$ on $W^u_{c}(x)$.
\end{definition}

\begin{definition}[Sets of stable/unstable continua]
Let $\diam(A)$ denotes the diameter of the set $A$ defined by $\diam(A)=\sup\{d(x,y); x,y\in A\}$,
\[
\mathcal{C}^s=\{C\in\mathcal{C}(X)\;;\;\diam(f^n(C))\to 0 \, \text{ when }\, n\to\infty\}, \,\,\,\,\,\, \text{and}
\]
\[
\mathcal{C}^u=\{C\in\mathcal{C}(X)\;;\;\diam(f^{-n}(C))\to 0\, \text{ when }\, n\to\infty\}.
\]
Continua in $\mathcal{C}^s$ are called stable and continua in $\mathcal{C}^u$ are called unstable. Let
\[
\mathcal{C}^s_\eps=\{C\in\mathcal{C}(X)\;;\;\diam(f^n(C))\leq\eps\, \text{ for every }\, n\geq 0\} \,\,\,\,\,\, \text{and}
\]
\[
\mathcal{C}^u_\eps=\{C\in\mathcal{C}(X)\;;\;\diam(f^{-n}(C))\leq\eps\, \text{ for every }\, n\geq 0\}.
\]
These sets contain exactly the $\eps$-stable and $\eps$-unstable continua of $f$, respectively.
\end{definition}

\begin{definition}[Continuum-wise expansiveness]
For each $x\in X$ and $c>0$ let $$\Gamma_{c}(x)=W^u_{c}(x)\cap W^s_{c}(x)$$ be the dynamical ball of $x$ with radius $c$. We say that $f$ is \emph{expansive} if there exists $c>0$ such that $$\Gamma_c(x)=\{x\} \,\,\,\,\,\, \text{for every} \,\,\,\,\,\, x\in X.$$ We say that $f$ is \emph{continuum-wise expansive} if there exists $c>0$ such that $\Gamma_{c}(x)$ is totally disconnected for every $x\in X$. The number $c$ is called a \emph{cw-expansive constant} of $f$.
\end{definition}
Cw-expansiveness was first considered by Kato in \cite{Kato1} and \cite{Kato2}. It was proved there that if $c>0$ is a cw-expansive constant of $f$ and $\eps<\frac{c}{2}$, then 
$$\mathcal{C}^s_{\eps}\subset\mathcal{C}^s \,\,\,\,\,\, \text{and} \,\,\,\,\,\, \mathcal{C}^u_{\eps}\subset\mathcal{C}^u.$$ The following result contains a mechanism that rules out the shadowing property for $C(f)$ when $f$ is a cw-expansive homeomorphism.

\begin{theorem}\label{cw-cf}
Let $f\colon X \rightarrow X$ be a cw-expansive homeomorphism of a compact metric space $(X,d)$. If for each $n\in\N$, there is a pair of non-trivial stable/unstable continua $(S_n,U_n)$ satisfying 
\begin{enumerate}
\item $d_H(S_n, U_n)\to0$ when $n\to\infty$ and 
\item there exists $r>0$ such that 
$$\diam(S_n)>r \,\,\,\,\,\, \text{and} \,\,\,\,\,\, \diam(U_n)>r \,\,\,\,\,\, \text{for every} \,\,\,\,\,\, n\in\N,$$
\end{enumerate}
then $C(f)$ does not have the shadowing property. 
\end{theorem}

\begin{proof}
Assume by contradiction that $C(f)$ has the shadowing property. Let $c\in(0,r)$ be a cw-expansive constant of $f$, $\eps\in(0,\frac{c}{4})$, and $\delta\in(0,\eps)$, given by the shadowing property of $C(f)$, be such that every $\delta$-pseudo-orbit of $C(f)$ is $\eps$-shadowed. Choose $n\in\N$ such that $d_H(S_n,U_n)<\delta$. The shadowing property of $C(f)$ ensures the existence of $C\in\mathcal{C}(X)$  satisfying
$$d_H(f^k(C),f^k(S_n))\leq\eps \,\,\,\,\,\, \text{for every} \,\,\,\,\,\, k\geq0 \,\,\,\,\,\, \text{and}$$
$$d_H(f^{-k}(C),f^{-k}(U_n))\leq\eps \,\,\,\,\,\, \text{for every} \,\,\,\,\,\, k\geq0.$$ Since $S_n\in\mathcal{C}^s$ and $U_n\in\mathcal{C}^u$, there is $k_0\in\N$ such that
$$\diam(f^k(S_n))\leq\eps \,\,\,\,\,\, \text{and} \,\,\,\,\,\, \diam(f^{-k}(U_n))\leq\eps \,\,\,\,\,\, \text{for every} \,\,\,\,\,\, k\geq k_0$$ and, consequently,
$$\diam(f^k(C))\leq2\eps \,\,\,\,\,\, \text{and} \,\,\,\,\,\, \diam(f^{-k}(C))\leq2\eps \,\,\,\,\,\, \text{for every} \,\,\,\,\,\, k\geq k_0.$$ It follows from the choice of $\eps$ that $$f^{k_0}(C)\in\mathcal{C}^s \,\,\,\,\,\, \text{and} \,\,\,\,\,\, f^{-k_0}(C)\in\mathcal{C}^u$$ and, consequently, $C\in\mathcal{C}^s\cap\mathcal{C}^u$. Note that $C$ is not a singleton since, in this case, $d_H(C,S_n)<\eps$ would imply $\diam(S_n)\leq\eps<r$, contradicting the assumption. 
By \cite[Corollary 5.5]{NADLER92} we choose a subcontinuum $C'\subset C$ such that $C'$ is not a singleton and has sufficiently small diameter, then \cite[Proposition 2.3.1]{ArtigueDend} ensures that $C'\in\mathcal{C}^s_{\eps}\cap\mathcal{C}^u_{\eps}$, contradicting the cw-expansiveness of $f$.
\end{proof}

\begin{remark}
It is important to note that in Theorem \ref{cw-cf}, the sequences $(S_n)_{n\in\N}$ and $(U_n)_{n\in\N}$ cannot be formed by $c$-stable and $c$-unstable continua, respectively. Indeed, in this case, any continuum that is accumulated by both $(S_n)_{n\in\N}$ and $(U_n)_{n\in\N}$ would be both $c$-stable and $c$-unstable, contradicting cw-expansiveness. This is more easily seen on Anosov diffeomorphisms, where local stable/unstable continua are contained in local stable/unstable manifolds and their transversality does not allow such an accumulation. 
\end{remark}

In what follows we prove the hypothesis of Theorem \ref{cw-cf} in some well known cases: among them are the Anosov diffeomorphisms and the continuum-wise hyperbolic homeomorphisms. To create the stable/unstable continua accumulating as in Theorem \ref{cw-cf}, we use global stable/unstable manifolds in the case of Anosov diffeomorphisms, while in the case of cw-hyperbolic homeomorphisms we use the global stable/unstable continua.

\begin{theorem}\label{dense}
Let $f\colon X \rightarrow X$ be a cw-expansive homeomorphism of a non-trivial compact metric space $(X,d)$, $c>0$ be a cw-expansive constant of $f$, and $\eps\in(0,\frac{c}{2})$. If there exist $x,y\in X$ such that
$$C^s(x):=\bigcup_{n\in\N}f^{-n}(C^s_{\eps}(f^n(x))) \,\,\,\,\,\, \text{and} \,\,\,\,\,\, C^u(y):=\bigcup_{n\in\N}f^{n}(C^u_{\eps}(f^{-n}(y)))$$
are dense in $X$, then $C(f)$ does not have the shadowing property.
\end{theorem}

\begin{proof}
It is enough to prove the existence of stable/unstable continua satisfying the hypothesis of Theorem \ref{cw-cf}. First, note that $C^s(x)$ and $C^u(y)$ are connected sets, since they are increasing unions of connected sets, and by hypothesis $\overline{C^s(x)}=X=\overline{C^u(y)}$. For each $n\in\N$, choose $k_n\in\N$ such that 
$$d_H\left(\overline{C^s(x)}, \bigcup_{i=1}^{k_n}f^{-i}(C^s_{\eps}(f^i(x)))\right)<\frac{\diam(X)}{n} \,\,\,\,\,\, \text{and}$$
$$d_H\left(\overline{C^u(y)}, \bigcup_{i=1}^{k_n}f^{i}(C^u_{\eps}(f^{-i}(y)))\right)<\frac{\diam(X)}{n}.$$
For each $n\in\N$, let
$$S_n=\bigcup_{i=1}^{k_n}f^{-i}(C^s_{\eps}(f^i(x))) \,\,\,\,\,\, \text{and} \,\,\,\,\,\, U_n=\bigcup_{i=1}^{k_n}f^{i}(C^u_{\eps}(f^{-i}(y))).$$
Thus, 
$$S_n\in\mathcal{C}^s \,\,\,\,\,\, \text{and} \,\,\,\,\,\, U_n\in\mathcal{C}^u \,\,\,\,\,\, \text{for every} \,\,\,\,\,\, n\in\N,$$
$$d_H(S_n,U_n)<\frac{2\diam(X)}{n}\to0 \,\,\,\,\,\, \text{when} \,\,\,\,\,\, n\to\infty$$
and
$$\diam(S_n)>\frac{\diam(X)}{2} \,\,\,\,\,\, \text{and} \,\,\,\,\,\, \diam(U_n)>\frac{\diam(X)}{2} \,\,\,\,\,\, \text{for every} \,\,\,\,\,\, n>1,$$ concluding the proof.
\end{proof}
We are able to prove the hypothesis of this theorem in the case of transitive cw-hyperbolic homeomorphisms.

\begin{definition}[Continuum-wise hyperbolicity]
We say that $f$ satisfies the $cw$-local-product-structure if for each $\eps>0$ there exists $\delta>0$ such that $$C^s_\eps(x)\cap C^u_\eps(y)\neq \emptyset \,\,\,\,\,\, \text{ whenever } \,\,\,\,\,\, d(x,y)< \delta.$$ The $cw$-expansive homeomorphisms satisfying the $cw$-local-product-structure are called $cw$-hyperbolic.
\end{definition}

For more information about cw-hyperbolicity, see \cite{ArrudaCarvalhoSarmiento2024}, \cite{ARTIGUE2024512} and \cite{CarvalhoRigo2023}.

\begin{definition}[Transitivity]
A map $f\colon X\to X$ is called \emph{transitive}, if for any pair $U,V\subset X$ of non-empty open subsets, there exists $n\in\N$ such that $$f^n(U)\cap V\neq\emptyset.$$
\end{definition}

The following is the main result of this section.

\begin{theorem}\label{cw-hyp} If $f\colon X\to X$ is a transitive cw-hyperbolic homeomorphism, then $C(f)$ does not have the shadowing property. In particular, if $M$ is a manifold of dimension at least two and $f\colon M\to M$ is a transitive Anosov diffeomorphism, then $C(f)$ does not have the shadowing property.
\end{theorem}

\begin{proof}
We will prove that $C^s(x)$ and $C^u(x)$ are dense in $X$ for every $x\in X$ (the proof is as in the case of topologically hyperbolic homeomorphisms \cite[Theorem 3.2.6]{AOKIHIRAIDE}). Theorem \ref{dense} ensures then that $C(f)$ does not have the shadowing property. Let $n\in\N$ be such that $f^n$ is topologically mixing (see \cite[Corollary 2.2 (4)]{ARTIGUE2024512}). To prove that $C^s(x)$ is dense in $X$, we consider any $z\in X$ and $\eps>0$, and prove the existence of a point in $C^s(x)\cap B(z,2\eps)$. Let $\delta>0$, given by the cw-local-product-structure of $f$, be such that 
$$C^s_\eps(x)\cap C^u_\eps(y)\neq \emptyset \,\,\,\,\,\, \text{ whenever } \,\,\,\,\,\, d(x,y)< 2\delta$$
and choose a finite set $\{x_1,\dots,x_j\}$ such that $\bigcup_{i=1}^j B(x_i,\delta)$ covers $X$. Since $f^n$ is topologically mixing, there exists $k\in\N$ such that $$f^{nk}(B(z,\eps))\cap B(x_i,\delta)\neq\emptyset \,\,\,\,\,\, \text{for every} \,\,\,\,\,\, i\in\{1,\dots,j\}.$$ Since $f^{nk}(x)\in B(x_i,\delta)$ for some $i\in\{1,\dots,j\}$ and there exists $w\in f^{nk}(B(z,\eps))\cap B(x_i,\delta)$, the cw-local-product-structure ensures the existence of $y\in C^s_\eps(f^{nk}(x))\cap C^u_\eps(w)$. Thus, 
$$f^{-nk}(y)\in f^{-nk}(C^s_\eps(f^{nk}(x)))\subset C^s(x)$$
and also
$f^{-nk}(y)\in C^u_{\eps}(f^{-nk}(w))$, which implies $f^{-nk}(y)\in B(z,2\eps)$ and consequently $$f^{-nk}(y)\in C^s(x)\cap B(z,2\eps).$$ Since this holds for each $\eps>0$ and each $z\in X$, it follows that $C^s(x)$ is dense in $X$. To prove that $C^u(x)$ is dense in $X$, we note that $f^{-1}$ is also a transitive cw-hyperbolic homeomorphism, so we can assume that $n\in\N$ was chosen such that $f^{-n}$ is topologically mixing and $k\in\N$ was chosen such that $$f^{-nk}(B(z,\eps))\cap B(x_i,\delta)\neq\emptyset \,\,\,\,\,\, \text{for every} \,\,\,\,\,\, i\in\{1,\dots,j\}.$$ Since $f^{-nk}(x)\in B(x_i,\delta)$ for some $i\in\{1,\dots,j\}$ and there exists $w'\in f^{-nk}(B(z,\eps))\cap B(x_i,\delta)$, the cw-local-product-structure ensures the existence of $y'\in C^u_\eps(f^{-nk}(x))\cap C^s_\eps(w)$. Thus, 
$$f^{nk}(y')\in f^{nk}(C^u_\eps(f^{-nk}(x)))\subset C^u(x)$$
and also
$f^{nk}(y')\in C^s_{\eps}(f^{nk}(w))$, which implies $f^{nk}(y')\in B(z,2\eps)$ and consequently $$f^{nk}(y')\in C^u(x)\cap B(z,2\eps).$$
Since this holds for each $\eps>0$ and each $z\in X$, it follows that $C^u(x)$ is dense in $X$ and concludes the proof. The statement for Anosov diffeomorphisms follows after noting that any Anosov diffeomorphism is cw-hyperbolic on a manifold with dimension at least 2.
\end{proof}

We point out that recently M. Carvalho et al. \cite[Theorem C]{carvalho2025sufficientconditionstransitivityhomeomorphisms} have shown if $f:M \rightarrow M$ is $\Omega$-stable $C^1$
 diffeomorphism of a smooth closed Riemannian manifold $M$
 with dimension at least 
2, then $f$ having barycenter property in $Per(f)$ implies that $f$ is a topologically transitive Anosov diffeomorphism. Hence, this result gives another method for constructing nice diffeomorphisms $f$ which have the shadowing property but the hyperspace mapping $C(f)$ does not have the shadowing property.

We also discuss the relation between transitivity of $f$ and of $C(f)$. It is proved in \cite{Peris} that $2^f$ is transitive if, and only if, $f$ is weakly mixing, that is, $f\times f$ is transitive. The case of $C(f)$ is not clear since $C(f)$ is not transitive when $f$ is defined on dendrites but it is transitive when $f$ is the shift map on $[0,1]^{\mathbb{Z}}$ \cite{AcostaGerardoIllanes}. We prove that $C(f)$ is not transitive when $f$ is a cw-expansive homeomorphism of a compact metric space with positive topological dimension.

\begin{theorem}
If $f\colon X\to X$ is a cw-expansive homeomorphism of a compact metric space with positive topological dimension, then $C(f)$ is not transitive.
\end{theorem}

\begin{proof}
This is a direct consequence of \cite[Proposition 2.2]{Kato1} which states that for every $\eps>0$ there exists $\delta>0$ such that if a continuum $C$ satisfies $\diam(C)\leq\delta$ and there exists $k\in\N$ such that $\diam(f^k(C))>\eps$, then 
$$\diam(f^n(C))\geq\delta \,\,\,\,\,\, \text{for every} \,\,\,\,\,\, n\geq k.$$ Let $A\subset X$ be a non-trivial continuum with $\eps=\diam(A)>0$ (which exists since $X$ has positive topological dimension) and choose $\delta\in(0,\eps)$ as in the result of Kato above for this $\eps$. Thus, if a continuum $C\subset X$ has a dense future orbit in $\mathcal{C}(X)$, then it has to accumulate at trivial continua of the form $\{x\}$, so there exists $k\in\N$ such that $\diam(f^k(C))<\delta$, and also at the non-trivial continuum $A$, so there exists $k'>k$ such that $\diam(f^{k'}(C))>\eps$. But since the future orbit of $C'$ accumulates at $\{x\}$, there exists $k''>k'$ such that $\diam(f^{k''}(C))<\delta$ contradicting Kato's result.
\end{proof}


\section{Dendrites and the shadowing property for \texorpdfstring{$C(f)$}{Cf}}

The goal of this section is two-fold: prove that for a dendrite monotone map $f:D \rightarrow D$, $f$ has the shadowing property if and only if $C(f)$ does and give algorithms which can be used to show that abundant of dendrites admit homeomorphisms with the shadowing property, including the universal dendrite of order $n>2$. We begin by stating necessary definitions.

\begin{definition}[Dendrites]
 A \emph{dendrite} is a Peano continuum which does not contain a simple closed curve.
\end{definition}

Dendrites enjoy a variety of characterizations and properties.  We refer  the reader to \cite[Chapter X]{NADLER92} for further general information on the topic. A characterization that will be useful for us is the following:
\begin{theorem}\label{thm:dendritecharacterization}\cite[Charactrization of Dendrites, Theorem 10.10]{NADLER92} 
A continuum $X$ is a dendrite if, and only if, the intersection of any two connected subsets of $X$ is connected.
\end{theorem}

We will also need pointwise behavior of dendrites, in particular the following definition. 
\begin{definition}[Order of a point]
    Let $X$ be a dendrite and $p\in X$. The \emph{order of $p$ in $X$} is the number of connected components of $X \setminus \{p\}$. We will say  that \emph{$p$ is a branch point of $X$} if the order of $p$ in $X$ is at least 3.
\end{definition}
Each point of a dendrite has a countable order. Moreover, the set of branch points of a dendrite is countable \cite[Theorem 10.23]{NADLER92}.
 
\begin{definition}[Monotone map]
A mapping $f:X \rightarrow X$ is said to be \emph{monotone} if $f^{-1}(p)$ is connected for all $p \in X$. When $X$ happens to be a compact metric space and $f$ is a surjection, this is equivalent to saying that $f^{-1}(A)$ is connected for all connected sets $A \subseteq X$ \cite[Ex 8.46]{NADLER92}.
\end{definition}

We refer the reader  to \cite[Chapter XIII]{NADLER92} for further information on monotone maps. Clearly, every homeomorphism is a monotone map, but we do not restrict ourselves to the case of invertible maps. Thus, it is important to note that in Definition \ref{shadowing} of the shadowing property, when $f$ is not invertible, we consider pseudo-orbits indexed by the positive integer numbers $\mathbb{N}$ while the shadowing orbit is a future orbit. It is a classical result in the shadowing theory that positive shadowing, finite shadowing, and shadowing are equivalent when the underlying space is compact (see \cite{PilyuginShadowingBook1999} for more details). We also recall the following theorem of Fernandez-Good that will be used in the proof of Theorem \ref{thm:dendrite} below.

\begin{theorem}\cite[Theorem 6]{FernandezGood}\label{FG}
If $X$ is a compact metric space, then $f:X \rightarrow X$ has the shadowing property if, and only if, $2^f: 2^X \rightarrow 2^X$ has the shadowing property.   
\end{theorem}

The following is one of the main results of this section.

\begin{theorem}\label{thm:dendrite} Let $D$ be a dendrite and $f:D \rightarrow D$ be a monotone map. Then, $f$ has the shadowing property if, and only if, $C(f)$ does. 
\end{theorem}

\begin{proof}
That the shadowing property for $C(f)$ implies the shadowing property for $f$ follows from \cite[Theorem 3.2]{FernandezGood}.

Let us now show the converse.
Let $\varepsilon >0$. By the fact that $D$ is locally connected, we may choose a finite cover $\mathcal{G}$ of $X$ consisting of open connected sets such that each $g \in \mathcal{G}$ has diameter less than $\varepsilon/4$. Let $\varepsilon'\in(0,\frac{\varepsilon}{4})$ be a Lebesgue number associated with the open cover $\mathcal{G}$, i.e., if $A \subseteq X$ has diameter less than $\varepsilon '$, then $A \subseteq g$ for some $g \in \mathcal{G}$. By Theorem~\ref{FG}, there is a $\delta >0$ which witnesses the fact that $2^f$ has the shadowing property with respect to $\varepsilon '$. 
We will show that the same $\delta$ witnesses that $C(f)$ has the shadowing property with respect to $\varepsilon$. To this end, let $(K_n)_{n \in \N}$ be a $\delta$-pseudo-orbit of $C(f)$ in $C(X)$. As $C(X) \subseteq 2^X$, we have that $(K_n)_{n \in \N}$ is a pseudo-orbit in $2^X$. By our choice of $\delta$, we have that there is a compact set $A \subseteq X$ such that 
$$d_H(f^n(A), K_n) < \varepsilon' \,\,\,\,\,\, \text{for every} \,\,\,\,\,\,n \in \N.$$ Now, for each $n\in\N$, let  \[ \mathcal{G}_n = \{g \in \mathcal{G}; g \cap K_n \neq \emptyset \}.   \]
We next observe that $\mathcal{G}_n$ covers $f^n(A)$. Indeed, let $x \in f^n(A)$. As $d_H(f^n(A), K_n) < \varepsilon ',$ there is an open ball $U$ with center $x$ with diameter less than $\varepsilon'$ which intersects $K_n$. As $\varepsilon'$ is a Lebesgue number of $\mathcal{G}$, there is $g \in \mathcal{G}$ such that $U \subseteq g$. As $g \cap K_n \neq \emptyset$, we have that $g \in \mathcal{G}_n$, implying that $\mathcal{G}_n$ covers $f^n(A)$. Now let
\[L_n = \overline{\bigcup_{g\in\mathcal{G}_n}g}.\]
We note that $L_n$ is a continuum as $K_n$ is a continuum and each $g \in \mathcal{G}_n$ is connected. Moreover, we have that
    \[f^n(A) \cup K_n \subseteq L_n \subseteq \bigcup_{x \in K_n} B(x, \varepsilon/4),
    \]
    implying that
\begin{eqnarray*}
d_H(f^n(A), L_n ) &\le& d_H(f^n(A), K_n ) + d_H(K_n, L_n )\\
&<& \varepsilon' + \frac{\varepsilon}{4}\\
&<& \frac{\varepsilon}{2}.
\end{eqnarray*}
   As $f$ is monotone, we have that $f^{-n}(L_n)$ is a continuum containing $A$. By Theorem~\ref{thm:dendritecharacterization}, we have that 
    \[K = \bigcap _{n \in \N} f^{-n}(L_n)\]
    is a continuum containing $A$. Now, as
    $$f^n(A) \subseteq f^n(K) \subseteq L_n \,\,\,\,\,\, \text{and} \,\,\,\,\,\, d_H(f^n(A), L_n ) < \frac{\varepsilon}{2},$$ we have that $d_H(f^n(K),f^n(A)) < \frac{\varepsilon}{2}$. Hence,
\begin{eqnarray*}
d_H(f^n(K), K_n) &\le& d_H(f^n(K), f^n(A)) + d_H(f^n(A), K_n) )\\ 
&<&\frac{\varepsilon}{2} + \varepsilon'\\
&<& \varepsilon,
\end{eqnarray*}
verifying that $K$ is a continuum which $\varepsilon$-shadows the pseudo-orbit $(K_n)_{n\in\N}$.
\end{proof}

We now explore which dendrites admit homeomorphisms with the shadowing property. 
The simplest possible dendrite is the interval $[0,1]$, so the above result ensures that homeomorphisms of $[0,1]$ with shadowing are also examples where $C(f)$ have shadowing. In what follows, let $\mathcal{H}([0,1])$ be the set of homeomorphisms $f\colon[0,1]\to[0,1]$ and $\mathcal{H^+}([0,1]), \mathcal{H^-}([0,1])$ denote the groups of homeomorphisms of $[0,1]$ which are increasing and decreasing, respectively. Using \cite{ArtigueCousillas2019}, we have the following corollary. 

\begin{corollary}\label{cor:typicalhomeoshadowing} A typical $h \in \mathcal{H}([0,1])$ has the property that $C(h)$ has the shadowing property.
\end{corollary}
\begin{proof} 
It was shown  \cite{ArtigueCousillas2019} that there is $h \in \mathcal{H^+}([0,1])$ whose conjugacy class is comeager and $h$ has the shadowing property. Using a symmetric argument, one can show there is a corresponding such map for $\mathcal{H^-}([0,1])$. Hence, we have that a typical $h  \in \mathcal{H}([0,1])$ has the shadowing property. Now as $[0,1]$ is a dendrite, Theorem~\ref{thm:dendrite} ensures that $C(h)$ has the shadowing property for a typical $h \in \mathcal{H} ([0,1])$.
\end{proof}
Given a dendrite $D$, it is well-known that a generic continuous self-map of $D$ has the shadowing property (see \cite{BrianMeddaughRaines, Koscielniaketal, Meddaugh}). Unfortunately, a generic map of the dendrite usually fails to be montone. The hypothesis of $f$ being monotone is important in the proof of Theorem \ref{thm:dendrite} and there are dendrites whose only homeomorphism is the identity map. Such dendrites can easily be  constructed by making branch points dense in the dendrite and the order of each branch point is distinct. In this article, we prove that a large class of dendrites admit homeomorphisms with the shadowing property. This will follow from a series of propositions we prove below. These propositions will give general methods for constructing dendrites which admit homeomorphisms with the shadowing property from simpler dendrites with such properties. Before we state our results, we first need the following definition. We use $\mathcal{H}(X)$ to denote the group of homeomorphisms of $X$.

\begin{definition}[Quasi-attractors]
Suppose that $X$ is a compact metric space and take $f\in\mathcal{H}(X)$. 
A compact $f$-invariant subset $A\subset X$ is a \emph{quasi-attractor} if for every 
open neighborhood $U$ of $A$ 
there is an open subset $V\subset U$ such that $A\subset V$ and 
${\overline{f(V)}}\subset V$. 
If, in addition, $f\colon A\to A$ has the shadowing property we say that $A$ is a \emph{quasi-attractor with the shadowing property}. We say that $A$ is a \emph{quasi-repeller} if it is a quasi-attractor for the inverse map $f^{-1}$.
\end{definition}

The following proposition \cite{ARTIGUE2024512} will be useful for us. 
\begin{proposition}\label{artigueprop3} \cite[Prop 3]{ARTIGUE2024512}
 If every point of $X$ belongs to a quasi-attractor with the shadowing property, then $f$ has the shadowing property.
\end{proposition}
The following proposition is a rather well-known fact. It also follows from the characterization of homeomorphisms of the interval with the shadowing property given in \cite{PenningsVaneeuwen1990}.
\begin{proposition}\label{basicinterval}
Consider $h:[0,1] \rightarrow [0,1]$ defined by $h(x) = x^2$. Then, $h$ has the shadowing property.
\end{proposition}

Next we define star union of countably many metric spaces with exactly one point in common.

\begin{definition}[Star union of metric spaces]
Let $((X_n,d_n))_{n\in\N}$ be a sequence of compact metric spaces and $p \in X_1$ be such that $X_n \cap X_m =\{p\}$ for all $n \neq m$. We define $\star(X_n,p)$ as the metric space $X = \cup_{n=1}^{\infty } X_n$ with the distance $d$ defined as follows:
\[d(a,b) = 
\begin{cases}
    d_n (a,b) & \textit { if } a,b \in X_n\\
    d_n(a,p)+d_m(b,p) & \textit { if } a \in X_n, \  b \in X_m, \ \textit{ and } n \neq m.
\end{cases}
\]
\end{definition}
The above definition is a small modification of a well-known metric space known as Kowalsky Hedgehog space \cite{SwardsonHedgehog}. In the standard hedgehog space $X_n$'s are taken to be the interval $[0,1]$.

Note that if $\diam(X_n) \rightarrow 0$ as $n \rightarrow \infty$, then $X$ is a compact metric space. In particular, we define the \emph{ $n$-star $\stn$} as  the set in the plane $\R^2$ consisting of $n$ intervals emanating from the origin  and \emph{ $\omega$-star $\stomega$} as the union of countably infinite number of intervals emanating from the origin whose diameters go to zero. More precisely, $\stomega$ can be taken as  $*(X_n,0)$ where $0$ is the origin in the plane and  $X_n =\{(t, \frac{1}{n} \cdot t):0 \le t \le \frac{1}{n}\}$.

\begin{proposition}\label{prop:starshadowing}
Suppose $\star(X_n,p)$ is as above with $\diam (X_n) \rightarrow 0$ and assume that $h_n:X_n \rightarrow X_n$ is a homeomorphism with the shadowing property such that  $\{p\}$ is a quasi-attractor fixed point of $h_n$ for every $n \in \N$. Then, $h:\star(X_n,p)\to \star(X_n,p)$ defined by $h(x) = h_n(x)$ when $x \in X_n$, is a homeomorphism with the shadowing property and $\{p\}$ is a quasi-attractor fixed point of $h$. 
\end{proposition}
\begin{proof}
As $p$ is a quasi-attractor of $h_n$ in $X_n$ for every $n \in \N$ and the $\diam (X_n) \rightarrow 0$, we have that $X_n \subseteq \star(X_n,p)$ is a quasi-attractor of  $h$. Thus, every point of $\star(X_n,p)$ is contained in a quasi-attractor with the shadowing property, so by Proposition~\ref{artigueprop3}, we have that $h$ has the shadowing property. That $p$ is a quasi-attractor of $\star(X_n,p)$ follows from the fact that $p$ is a quasi-attractor of each $X_n$ and the diameters of $X_n$ goes to zero as $n$ goes to infinity. 
\end{proof}
\begin{corollary}\label{cor:nstarshadowing}
For each $n\in\N\cup\{\omega\}$, the n-star $\stn$  admits homeomorphisms with the shadowing property and having the origin as a quasi-attracting fixed point.
\end{corollary}
\begin{proof}
    This simply follows from Propositions~\ref{prop:starshadowing} and \ref{basicinterval}.
\end{proof}

\begin{definition}[Bridge space]
For $i =1, 2$, let $(X_i,d_i)$ be disjoint metric spaces and $p_i \in X_i$. Then, $X_1 \leftrightsquigarrow X_2$ is the bridge space of $X_1,X_2$ defined as follows:

\[ X_1 \leftrightsquigarrow X_2 = X_1 \cup [p_1, p_2] \cup X_2,\]
where $[p_1,p_2]$ is homeomorphic to an interval with endpoints $p_1,p_2$ such that $ [p_1, p_2] \cap X_i = \{p_i\} $, $i =1,2$. Without loss of generality, we assume that the metric $d'$ on $[p_1,p_2]$ is that of the interval $[0,1]$. The metric $d$ on $X_1 \leftrightsquigarrow X_2$ is defined as follows
\[d(a,b) = 
\begin{cases}
    d_i (a,b) & \textit { if } a,b \in X_i\\
    d'(a,b) & \textit{ if } a,b \in [p_1,p_2]\\
    d_i(a,p_i)+d'(b,p_i) & \textit { if } a \in X_i, \  b \in [p_1,p_2]\\
    d_1(a,p_1)+d_2(b,p_2) +1  & \textit { if } a \in X_1, \  b \in X_2.
\end{cases}
\]
\end{definition}
It is clear that $d(a,b)\geq0$ and that $d(a,b)=d(b,a)$. The triangle inequality for $d$ is proved as follows: we deal with the case $a\in X_1$, $b\in [p_1,p_2]$, and $c\in X_2$;  other cases are analogous.
\begin{eqnarray*}
d(a,c)&=&d_1(a,p_1)+d_2(c,p_2)+1\\
&=& d_1(a,p_1)+d'(b,p_1)+d'(b,p_2)+d_2(c,p_2)\\
&=& d(a,b)+d(b,c).
\end{eqnarray*}

\begin{proposition}\label{prop:bridge}
Suppose $X_1 \leftrightsquigarrow X_2$ is the bridge space of $X_1,X_2$ as defined above. Furthermore, assume that $\{p_i\}$ is a quasi-attractor of some map $h_i:X_i \rightarrow X_i$ with the shadowing property. Let $h_3:[p_1,p_2] \rightarrow [p_1,p_2]$ be any homeomorphism with the shadowing property such that $h_3(p_i) = p_i$ and $\{p_1,p_2\}$ is a quasi-attractor set of $h_3$. Then, the map $h\colon X_1 \leftrightsquigarrow X_2 \rightarrow X_1 \leftrightsquigarrow X_2$ defined by
  \[h(x) = 
\begin{cases}
    h_i(x) & \textit { if } x \in X_i, \ i=1, 2 \\
    h_3(x) & \textit { if } x \in [p_1,p_2]
\end{cases}
\]  
    has the shadowing property. 
    \end{proposition}
\begin{proof}
Note that each $X_i$, with $i=1,2$, and $[p_1,p_2]$ are quasi-attractors of $h$ with shadowing. Then the result follows from Proposition \ref{artigueprop3}.  
\end{proof}
\begin{remark}\label{rem:homeoattractringends}
A map $h_3$ with the shadowing property as required in Proposition~\ref{prop:bridge}, can be easily constructed using the characterization of shadowing given in \cite{PenningsVaneeuwen1990}. For example, take $h_3$ conjugate to a homeomoprhism $f$ of $[0,1]$ whose fixed point set is $\{0,1/2,1\}$ such that $f$ is below the identity on $[0,1/2]$ and above the identity on $[1/2,1]$.
\end{remark}
\begin{corollary}\label{cor:bridgeshadowing} Suppose $h_i:X_i \rightarrow X_i$, $1 \le i \le 2$, are homeomorphisms with the shadowing property with $p_i\in X_i$ quasi-attractor fixed point of $h_i$. Then, the bridge space $X_1 \leftrightsquigarrow X_2$ admits a homeomorphism with the shadowing property.  
\end{corollary}
\begin{proof}
    This simply follows from Proposition~\ref{prop:bridge} and Remark~\ref{rem:homeoattractringends}.
\end{proof}

\begin{definition}[Combs] Let $X$ be a continuum, $D = \{d_1, d_2,\ldots \} \subset X$ be a countable set, and $1 \le n \le \omega$. We define the \emph {$(X,D,n )$-comb} as the following subset of $X \times \R^2 :$ 
\[Z =  \{ (x,0); x \in X \setminus D \} \cup \{(d_i, 1/i\cdot y); d_i \in D, y \in \stn \}.
\]
Above, $0$ denotes the origin in $\R^2$. We endow $X \times \R ^2$ with the product metric given by the maximum between the metric on $X$ and the metric on $\R^2$.
\end{definition}

We note that $Z$ is a closed subset of $X \times \R^2$, and hence a continuum, as $X$ and $\stn$ are continua. One can think of $Z$ as $X$ with a decreasing sequence of stars $\stn$ attached at each point of $D$. We note that there is a natural projection map $\Pi: Z \rightarrow X$ defined by $\Pi (x,y ) = x$.
\begin{definition}\label{simple}
    We call a dynamical system {\em $(X,f)$ simple} if it has a quasi-attractor fixed point $p$ and a quasi-repeller fixed point $q$ such that for all $x \in X \setminus \{p,q\}$ we have that $lim_{n \rightarrow \infty} f^{n}(x) = p$ and $lim_{n \rightarrow \infty} f^{-n}(x) = q$.
\end{definition}

The following proposition is well-known when $(X,f)$ is a simple Morse-Smale diffeomorphism of a manifold \cite{Palis} or an increasing/decreasing continuous function of the interval \cite{PenningsVaneeuwen1990}. We could not find a proof in the general case of homeomorphisms of compact metric spaces, so we include a proof for the sake of completeness.
\begin{proposition}\label{simpleimplyshadowing} If $(X,f)$ is a simple homeomorphism of a compact metric space, then $f$ has the shadowing property. 
\end{proposition}

\begin{proof} Let $(X,f)$ be a simple homeomorphism with quasi-attractor fixed point $p$, quasi-repeller fixed point $q$, and $\varepsilon >0$. Let $U_p, U_q$ be open sets  with diameters less than $\varepsilon/2$ containing $p, q$, respectively, such that $\overline{f(U_p)} \subseteq U_p$, $\overline{f^{-1}(U_q)} \subseteq U_q$, and $U_p \cap U_q = \emptyset$.  As $p$ is a quasi-attractor fixed point of $f$, for each $x \in X \setminus (U_p \cup f^{-1}(U_q) )$, there is a positive integer $N_x$ such that $f^{N_x} (x) \in f(U_p)$. Note that for $n > N_x$, $f^n(x) \in f(U_p)$. Using the compactness of $X \setminus (U_p \cup f^{-1}(U_q) ) $ and the continuity of $f$, we may assume that there is a positive integer $N$ such that $N_x < N$ for all $x \in X \setminus (U_p \cup f^{-1}(U_q) )$. 

Let 
$\eta >0$ be such that 
$$\eta < 1/2  \cdot \min \{ \varepsilon, d_H(\overline{f(U_p)}, X \setminus U_p), d_H(\overline{f^{-1} (U_q)}, X \setminus U_q)\}.$$ Note that if $x \in f(U_p)$ and $y$ is $\eta$-close to $x$, then $f(y) \in f(U_p)$. Analogously, if $x \in f^{-1}(U_q)$ and $y$ is $\eta$-close to $x$, then $f^{-1}(y) \in f^{-1}(U_q)$. Using these observations, we have that if $(x_n)_{n=0}^{\infty}$ is any $\eta$-pseudo-orbit of $f$ with $x_0 \in U_p$, then $x_n \in U_p$ for all $n \ge 0$. Analogously, if $(x_n)_{n=0}^{\infty}$ is any $\eta$-pseudo-orbit  of $f^{-1}$ with $x_0 \in U_q$, we have that $x_n \in U_q$ for all $n \ge 0$.

Let $0 < \delta < \eta$ be such that if $x_0 \in X \setminus (U_p \cup f^{-1}(U_q))$ and $(x_j)_{j=0}^N$ is any $\delta$-pseudo-orbit of $f$, then $d(x_j, f^j(x_0))< \eta$ for all $ 0 \le j \le N$. This may be done using the uniform continuity of $f, \ldots, f^N$. Moreover, we also require that if $d(x,y)  < \delta$, then $d(f^{-1}(x), f^{-1}(y)) < \eta$.

Now, let $(x_j)_{j=0}^{\infty}$ be a $\delta$-pseudo-orbit of $f$. We need to show that it is $\varepsilon$-shadowed  by a real orbit. If $(x_j)_{j=0}^{\infty} \subseteq U_p$, then the fixed point $p$ $\varepsilon$-shadows $(x_j)_{j=0}^{\infty}$. Similarly, if $(x_j)_{j=0}^{\infty} \subseteq U_q$, then the fixed point $q$ $\varepsilon$-shadows $(x_j)_{j=0}^{\infty}$. Otherwise, we may choose the least $n \ge 0$ such that $x_n \notin  f^{-1}(U_q)$.  As $\delta$ is sufficiently small, it follows that $x_n \notin U_p$. As $d(f(x_{n-1}),x_n) < \delta$, by our choice of $\delta$ we have that $d(x_{n-1}, f^{-1}(x_n)) < \eta$. Since $x_{n-1} \in f^{-1}(U_q)$, by our choice of $\eta$ we have that $f^{-1}(x_n) \in U_q$. Hence, for all $k \ge 1$, we have that $f^{-k}(x_n) \in U_q$. Now it is easy to verify that $f^{-n}(x_n)$ $\varepsilon$-shadows $(x_j)_{j=0}^{\infty}$. Indeed, for $0 \le j <n$, we have that $x_j$ and $f^j(f^{-n}(x_n))$ belong to $U_q$. For $ n \le j \le n+N$ we have that $d(x_j,f^j(f^{-n}(x_n))) < \varepsilon$ by our choice of $\delta$. Finally, for $j > n+N$, we have that $x_j$ and $f^j(f^{-n}(x_n))$ both belong to $U_p$.
\end{proof}

\begin{proposition}\label{prop:comb}
Let $X$ be a continuum, $D = \{d_1, d_2,\ldots \} \subseteq X$ be a countable set, and $1 \le n \le \omega$. Let $h_1\colon X \rightarrow X$ be such that  $(X,h_1)$ is a simple dynamical system with quasi-attractor and quasi-repeller fixed points $p,q$, respectively, with $p,q \in X \setminus D$, and $h_1(D)=D$. Then, there is a homeomorphism $h\colon Z \rightarrow Z$,  $Z=(X,D,n)$-comb, such that
\begin{itemize}
    \item $\Pi \circ h = h_1 \circ \Pi$ where $\Pi:Z \rightarrow X$ is the projection map, and 
    \item $(Z,h)$ is a simple dynamical system with quasi-attractor and quasi-repeller points $(p,0)$ and $(q,0)$, respectively.
\end{itemize}
\end{proposition}

\begin{proof} 
Let $h_2\colon \stn \rightarrow \stn$ be a homeomorphism such that $h_2(0) =0$ and define $h\colon Z \rightarrow Z$ on $Z=(X,D,n)$-comb by 
\[ 
   h (x,y) = \begin{cases}
       (h_1(x), 0) & \textit{ if } x \notin D\\
       (h_1(x), \frac{1}{j} \cdot h_2( i \cdot y)) & \textit{ if } x = d_i, \  h_1(x) = d_j.
   \end{cases}
\]
It is easy to check that $\Pi \circ h = h_1 \circ \Pi$. Moreover, by definition, we have that $\Pi$ is a monotone map.

As $p,q$ are quasi-attractor and quasi-repeller points of $h_1$ and for each $x \in D$ we see that $(h_1^{n}(x))_{n \in \Z}$ is infinite, it follows that the diameter of ``teeth'' attached at $h_1^n(x)$  goes to zero when $n \rightarrow \pm \infty$. Hence, we have that $(p,0)$ and $(q,0)$ are quasi-attractor and quasi-repeller points of $h$, respectively. More specifically,  we have that $(h^n(x,y))_{n\in\Z}$ converges to $(p,0)$ when $n\to\infty$ and to $(q,0)$ when $n\to-\infty$ for every $(x,y)\in Z\setminus\{(p,0),(q,0)\}$, and consequently $(Z,h)$ is a simple system. 

\end{proof}
\begin{remark}
 Using Propositions~\ref{prop:starshadowing}, \ref{prop:bridge}, and \ref{prop:comb}, one  can construct numerous dendrites and their homeomorphisms with the shadowing property. This can be done, for example, by bridging stars and forming combs, as in the proof of Theorem \ref{UniversalDendrite} below. 
\end{remark}
   
Our next goal is to show that the universal dendrite of order $n$ admits homeomorphisms with the shadowing property. 
\begin{definition}[Universal dendrite]
The \emph{universal dendrite of order $n$}, $n \ge 3$, denoted by $U_n$, is a dendrite such that the set of branch points of $U_n$ is dense in $U_n$ and each branch point has order $n$.
The arc can be considered as $U_2$ since every point in it except the endpoints have order $2$, and $U_1$ is simply the degenerate continuum consisting of a single point.
\end{definition}
There is only one universal dendrite of order $n$ up to homeomorphism \cite{Charatonik1994}. Moreover, any dendrite whose branch points have order $n$ or less can be embedded into $U_n$, hence the name universal \cite{NADLER92}.

We next recall some results concerning the inverse limit of dynamical systems (see \cite{GoodMitchellThomasPreservShadowing2020, BernardesShadwoingTopApp2023}). 

\begin{definition}[Inverse limit systems]
Let $((X_i,h_i))_{i\in\N}$ be a sequence (possibly finite) of topological dynamical systems and for each $i\in\N$ let $\varphi_i:X_{i+1} \rightarrow X_i$ be a surjective continuous map such that $h_i\circ \varphi_{i} =\varphi_{i} \circ h_{i+1}$. We define the inverse limit of $((X_i,h_i))_{i\in\N}$ as the set
\[\varprojlim \{X_i, \varphi_i\} :=  \left \{ (x_1,x_2,\ldots) \in \prod_{i\in\N} X_i;\,\, \varphi_i(x_{i+1}) = x_i  \right \} \] and the induced dynamical system on $\varprojlim \{X_i, \varphi_i\}$
\[h:\varprojlim \{X_i, \varphi_i\} \rightarrow \varprojlim \{X_i, \varphi_i\} \] by 
     \[h(x_1,x_2,\ldots) = (h_1(x_1), h_2(x_2), \ldots).\]
\end{definition}
    
We have the following well-known result (see \cite{GoodMitchellThomasPreservShadowing2020, BernardesShadwoingTopApp2023} for information and generalizations).
\begin{proposition}\label{prop:invlimitshadowing} Suppose we are in the setting above. If each of $h_i$ has the shadowing property, then so does $h$.
\end{proposition}

The following is the last result of this article. $U_2$ can be considered an arc as all points besides the end 

\begin{theorem}\label{UniversalDendrite} The universal dendrite $U_n$ admits a homeomorphism with the shadowing property for every $1 \le n \le \infty$.
\end{theorem}
\begin{proof}
The case $n=1$ and $n=2$ are trivial as $U_1$ is a single point and $U_2$ is an arc which admits many homeomorphisms with the shadowing property. Hence we consider $n \ge 3$.
Let $Y$ be the $(n-2)$-star in the plane emanating from the origin $0$. By Corollary~\ref{cor:nstarshadowing}, there is a homeomorphism $f:Y \rightarrow Y$ with the shadowing property such that $f(0)=0$. We will construct $U_n$ with an induction process which we begin by letting $X_0 = [0,1]$, $h_0:X_0 \rightarrow X_0$ be a simple homeomorphism with quasi-attractor and quasi-repeller fixed points $0,1$, and $D_0$ be a countable dense subset of $(0,1)$ such that $h_0(D_0) =D_0$.

Suppose we are at step $k \ge 1$, a dendrite $X_k$, a simple homeomorphism $h_k:X_k \rightarrow X_k$ with quasi-attractor point $p_k$ and quasi-repeller point $q_k$, and a countable dense set $D_k \subseteq X_k \setminus \{p_k, q_k\}$ with $h_k(D_k) = D_k$ are defined so that \begin{itemize}
        \item the set of branch points of $X_k$ does not intersect $D_k$, 
        \item each $x \in D_k$ has order 2 in $X_k$,
        \item each branch point of $X_k$ has  order $n$.
        
    \end{itemize} Applying Proposition~\ref{prop:comb} to $X_k$, $h_k$, $D_k$, $Y$, and $f$, we obtain $X_{k+1}$, $h_{k+1}$, and a surjective map $\varphi_k : X_{k+1} \rightarrow X_k$ such that
    \begin{itemize}
        \item $X_{k+1}=(X_k,D_k,n-2)$-comb,
        \item $h_{k+1}:X_{k+1} \rightarrow X_{k+1}$ is a simple homeomorphism with quasi-attractor fixed point $(p_k,0)$ and quasi-repeller fixed point $(q_k,0)$, 
        \item $\varphi_k : X_{k+1} \rightarrow X_k$ is a monotone map such that $\varphi_{k} \circ h_{k+1} = h_{k} \circ \varphi_k$, and 
        \item each point of $D_k\times \{0\}$ is a branch point of $X_{k+1}$ of order $n$.
    \end{itemize}
We now let $D_{k+1} \subseteq X_{k+1}$ be any countable set dense in $X_{k+1}$  such that for all $(x,y) \in D_{k+1}  $, $x \in D_k$, $(x,y)$ is a point of order 2 of $X_{k+1}$ and $h_{k+1}(D_{k+1}) = D_{k+1}$. This can easily be accomplished by choosing a countable set dense in each ``tooth'' of the comb and closing the union of these sets under the orbit of $h_{k+1}$. 

Let us now see why the induction hypotheses hold at step $k+1$. Indeed, by our choice of $D_{k+1}$, we have that no point of $D_{k+1}$ is a branch point of $X_{k+1}$ and each point of $D_{k+1}$ has order $2$ in $X_{k+1}$. To see that each branch point of $X_{k+1}$ has order $n$, 
note that the branch points of $X_{k+1}$ are of the form $(p,0)$, with $p \in X_k$. If $p \notin D_k$, then $p$ must be a branch point of $X_k$ and $p$ has order $n$ in $X_k$. As nothing is attached at $(p,0)$ when forming $X_{k+1}$, we have that $(p,0)$ has order $n$ in $X_{k+1}$. If $p \in D_k$, then by the induction hypotheses at step $k$, $p$ has order $2$ in $X_k$.  As $X_{k+1}$ is formed by attaching  $\mbox{st}_{n-2}$ at $p$, we have that $(p,0)$ has order $n$ in $X_{k+1}$. This finishes the induction step.

Now we consider the inverse limit space $X:=\varprojlim \{X_i, \varphi_i\}$, which is a continuum. Indeed, as $\varphi_i$'s are monotone, we have that $X$ is a dendrite (see \cite[Theorem 10.36]{NADLER92}). 
Now consider the induced map $h\colon X \rightarrow X$ given by $h(x_1,x_2,\ldots) = (h_1(x_1), h_2(x_2), \ldots)$. By Proposition~\ref{simpleimplyshadowing}, we have that each $h_k$ has the shadowing property. By Proposition~\ref{prop:invlimitshadowing}, we have that $h$ is a homeomorphism of $X$ with the shadowing property. It only remains to show that $X$ is the universal dendrite of order $n$, that is, each branch point of $X$ has order $n$ and the set of branch points of $X$ is dense in $X$. 

We let $\Pi_k$ be the projection of $X$ onto the $X_k$ coordinate. As each $\varphi_i$ is  monotone, $\Pi_k$ is monotone \cite[Ex 8.47]{NADLER92}.
    We will first show that each branch point of $X$ has order $n$, or, equivalently, the order of each point of $X$ is either 1, 2, or $n$.  To this end, consider $(x_0, x_1, \ldots) \in X$. We observe that by construction we have that if $x_k \notin D_k$ for some $k\in\N$, then $x_{k+1} = (x_k,0)$ and $x_{k+1} \notin D_{k+1}$. Hence, for such $k$, we have that $x_l \notin D_l$ for every $l \ge k$.  Thus, we have two cases to consider, either there exists $k\in\N$ such that $x_l \notin D_l$ for every $l \ge k$, or $x_k \in D_k$ for every $k \in \N$.
    In the first case, by hypothesis, we have that $x_k$ has order $j$ in $X_k$ where $j \in \{1, 2, n \}$. As such, we have $X_k \setminus \{x_k\}$ has $j$ components, labeled $A_1, \ldots, A_j$. As $\Pi_k$ is  monotone, we have that each $\Pi_k^{-1}(A_i)$ is connected in $X$.  As $x_l \notin D_l$ for all $l \ge k$, we have that $\varphi_l^{-1}(x_l)$ has exactly one element for each $l \ge k$. Hence, we have that 
    $$\Pi_k^{-1}(x_k) = (x_0, x_1, \ldots, x_k, \ldots),$$ i.e, $X \setminus \{ (x_0, x_1, 
    \ldots) \}$ is the union of $j$ connected sets, namely, $\Pi_k^{-1}(A_i)$, $ 1 \le i \le j$. This ensures that $(x_0, x_1, \ldots, x_k, \ldots)$ has order $j$ in $X$. Let us now consider the case  $x_k \in D_k$ for all $k \in \N$. In this case, it suffices to show that $X \setminus \{(x_0,x_1,\ldots)\}$ has a component dense in $X$, implying that the order of $(x_0,x_1, \ldots)$ is 1 in $X$ as $X$ is a dendrite. The density follows from the fact that for each $k\in\N$, there is a component of $X_{k+1} \setminus \{x_{k+1}\}$ which contains $X_k \times \{0\}$ as a subset. This is the case as all elements of $D_{k+1}$ are chosen from the ``teeth'' of $X_{k+1}$.

    Finally, to see that the set of branch points of order $n$ is dense in $X$, we observe that for all $k\in\N$ and all $x \in D_k$, $(x,0)$ is a branch point of order $n$ of $X_{k+1}$ and hence if we consider a point $(x_0, x_1, \ldots) \in X$ such that $x_k \in D_k$ and $x_{l+1} = (x_l,0)$ for all $l \ge k$, then $(x_0, x_1, \ldots) \in X$ is a branch point of order $n$ in $X$. Hence, for all $k \in \N$ and all open subsets $U \subseteq X_k$, we see that $[U] := \{(x_1, x_2,\ldots) \in X; x_k \in U\}$ contains a branch point of $X$ of order $n$, implying that the set of branch points of order $n$ of $X$ is dense in $X$ and completes the proof. 
\end{proof}

\section{Open questions}

The examples of systems with shadowing for $C(f)$ in this article involve dendrites, one-dimensional continua.  The other known result concerning the shadowing property for $C(f)$ is negative \cite{ArbietoBohorquez}, leading to the following question.

\begin{question}
Does there exist a topological dynamical system $(X,f)$, $X$ a continuum with dimension greater than one, such that $C(f)$ has the shadowing property?
\end{question}

We still cannot decide whether $C(f)$ has the shadowing property in the case of general Morse-Smale diffeomorphisms. The following question was stated in \cite{ArbietoBohorquez} and is still open:

\begin{question}
Does there exist a Morse-Smale diffeomorphism $f\colon M\to M$ of a closed manifold such that $C(f)$ has the shadowing property?
\end{question}





\vspace{+0.8cm}

\hspace{-0.45cm}\textbf{Acknowledgments.}
Bernardo Carvalho was supported by Progetto di Eccellenza MatMod@TOV (CUP E83C23000330006) and Prin 2022 (PRIN 2017S35EHN). The authors thank Prof. Carlangelo Liverani for supporting two visits of U. Darji to the Mathematics Department of the University of Rome Tor Vergata where this research was developed. The hospitality of the department is especially appreciated. The authors thank A. Kocsis and T. Katay  for observing an error in the proof of Proposition \ref{prop:comb} of a previous version of this article. The authors also thank M. Elekes and M. Palfy for ideas that led to a corrected version in the present article.\\

The authors thank anonymous referees whose valuable suggestions improved the exposition of the article.

\bibliographystyle{acm}
\bibliography{biblio}

\vspace{1.0cm}
\noindent

{\em B. Carvalho}
\vspace{0.2cm}

\noindent

Dipartimento di Matematica,

Università degli Studi di Roma Tor Vergata

Via Cracovia n.50 - 00133

Roma - RM, Italy
\vspace{0.2cm}

\email{mldbnr01@uniroma2.it}

\vspace{1cm}
{\em U. B. Darji}
\vspace{0.2cm}

\noindent

Department of Mathematics,

University of Louisville

Louisville, KY 40292
USA
\vspace{0.2cm}

\email{ubdarj01@louisville.edu}
\end{document}